\documentclass{scrarticle}
\usepackage{braket,amsfonts}

\usepackage{array}

\usepackage[caption=false]{subfig}

\usepackage{tikz, pgfplots}
\pgfplotsset{compat=1.18}

\usetikzlibrary{arrows.meta}
\usetikzlibrary{backgrounds}
\usepgfplotslibrary{patchplots}
\usepgfplotslibrary{fillbetween}
\pgfplotsset{%
    layers/standard/.define layer set={%
        background,axis background,axis grid,axis ticks,axis lines,axis tick labels,pre main,main,axis descriptions,axis foreground%
    }{
        grid style={/pgfplots/on layer=axis grid},%
        tick style={/pgfplots/on layer=axis ticks},%
        axis line style={/pgfplots/on layer=axis lines},%
        label style={/pgfplots/on layer=axis descriptions},%
        legend style={/pgfplots/on layer=axis descriptions},%
        title style={/pgfplots/on layer=axis descriptions},%
        colorbar style={/pgfplots/on layer=axis descriptions},%
        ticklabel style={/pgfplots/on layer=axis tick labels},%
        axis background@ style={/pgfplots/on layer=axis background},%
        3d box foreground style={/pgfplots/on layer=axis foreground},%
    },
}

\usepackage{amsmath}
\usepackage{amssymb}
\usepackage{xcolor}

\usepackage[breaklinks=true]{hyperref}

\colorlet{RefColor}{green!50!black}
\colorlet{LinkColor}{red!50!black}
\hypersetup{
  colorlinks = true,
  linkbordercolor = {white},
  linkcolor = LinkColor,
  anchorcolor=LinkColor,
  citecolor= RefColor,
  filecolor=cyan,
  menucolor=red,
  runcolor=cyan,
  urlcolor = LinkColor,
}

\usepackage{algorithm}
\usepackage{algpseudocode}

\usepackage{graphicx,epstopdf}
\DeclareGraphicsExtensions{.pdf,.eps,.png,.jpg,.jpeg}

\usepackage{amsopn}
\usepackage{amsthm}
\usepackage{amsmath}

\newtheorem{proposition}{Proposition}
\newtheorem{corollary}{Corollary}
\newtheorem{remark}{Remark}
\usepackage{xspace}
\usepackage{bold-extra}
\usepackage[most]{tcolorbox}
\usepackage[font={small,it}]{caption}

\colorlet{texcscolor}{blue!50!black}
\colorlet{texemcolor}{red!70!black}
\colorlet{texpreamble}{red!70!black}
\colorlet{codebackground}{black!25!white!25}

\DeclareOldFontCommand{\sc}{\normalfont\scshape}{\@nomath\sc}

\colorlet{header1}{blue!10!black}
\newcommand\funding[1]{\protect\\ \hspace*{1.8em}{\color{header1}\bfseries Funding:} #1}
\newcommand{\email}[1]{\protect\href{mailto:#1}{#1}}
\newcommand\keywordsname{Key words}
\newcommand\keywordname{Key word}
\newcommand\MSCcodesname{MSC codes}
\newcommand\MSCcodename{MSC code}
\newenvironment{@abssec}[1]{%
     \if@twocolumn
       \section*{#1}%
     \else
       \vspace{.05in}\footnotesize
       \parindent .2in
         \@hangfrom{\color{header1}\bfseries #1. }\ignorespaces 
     \fi}
     {\if@twocolumn\else\par\vspace{.1in}\fi}
\newenvironment{keywords}{\begin{@abssec}{\keywordsname}}{\end{@abssec}}

\newenvironment{AMS}{\begin{@abssec}{MSC codes}}{\end{@abssec}}

\lstdefinestyle{siamlatex}{%
  style=tcblatex,
  texcsstyle=*\color{texcscolor},
  texcsstyle=[2]\color{texemcolor},
  keywordstyle=[2]\color{texemcolor},
  moretexcs={cref,Cref,maketitle,mathcal,text,headers,email,url},
}

\tcbset{%
  colframe=black!75!white!75,
  coltitle=white,
  colback=codebackground, %
  colbacklower=white, %
  fonttitle=\bfseries,
  arc=0pt,outer arc=0pt,
  top=1pt,bottom=1pt,left=1mm,right=1mm,middle=1mm,boxsep=1mm,
  leftrule=0.3mm,rightrule=0.3mm,toprule=0.3mm,bottomrule=0.3mm,
  listing options={style=siamlatex}
}

\newtcblisting[use counter=example]{example}[2][]{%
  title={Example~\thetcbcounter: #2},#1}

\newtcbinputlisting[use counter=example]{\examplefile}[3][]{%
  title={Example~\thetcbcounter: #2},listing file={#3},#1}

\DeclareTotalTCBox{\code}{ v O{} }
{ %
  fontupper=\ttfamily\color{black},
  nobeforeafter,
  tcbox raise base,
  colback=codebackground,colframe=white,
  top=0pt,bottom=0pt,left=0mm,right=0mm,
  leftrule=0pt,rightrule=0pt,toprule=0mm,bottomrule=0mm,
  boxsep=0.5mm,
  #2}{#1}

\patchcmd\newpage{\vfil}{}{}{}
\title{Empirical sparse regression on quadratic manifolds\thanks{\funding{This material is based upon work supported by the U.S.~Department of Energy, Office of Science Energy Earthshot Initiative as part of the project "Learning reduced models under extreme data conditions for design and rapid decision-making in complex systems" under Award \#DE-SC0024721. This work was additionally supported in part by the Air Force Center of Excellence on Multi-Fidelity Modeling of Rocket Combustor Dynamics, award FA9550-17-1-0195. P.S.~was additionally partially supported by the Air Force Office of Scientific Research (AFOSR) award FA9550-21-1-0222 (Dr.~Fariba Fahroo).}}}

\author{Paul Schwerdtner\thanks{Corresponding author. Courant Institute of Mathematical Sciences, New York University, New York, NY 10012, USA;  \email{paul.schwerdtner@nyu.edu}} \and Serkan Gugercin\thanks{Department of Mathematics and Division of
Computational Modeling and Data Analytics, Academy of Data Science,
Virginia Tech, Blacksburg, VA 24061, USA} \and Benjamin Peherstorfer\thanks{Courant Institute of Mathematical Sciences, New York University, New York, NY 10012, USA}}

\usepackage{amsopn}

\input{Notation.sty}

\ifpdf
\hypersetup{
  pdftitle={Empirical sparse regression on quadratic manifolds},
  pdfauthor={}
}
\fi

\let\oldfigure\figure
\let\endoldfigure\endfigure
\renewenvironment{figure}
  {\oldfigure\small}
  {\endoldfigure}

\graphicspath{{PlotSources/final}}

\begin{document}

\maketitle

\begin{abstract}
Approximating field variables and data vectors from sparse samples is a key challenge in computational science. Widely used methods such as gappy proper orthogonal decomposition and empirical interpolation  rely on linear approximation spaces, limiting their effectiveness for data representing transport-dominated and wave-like dynamics. To address this limitation, we introduce quadratic manifold sparse regression, which trains quadratic manifolds with a sparse greedy method and computes  approximations on the manifold through novel  nonlinear projections of sparse samples. The nonlinear approximations obtained with quadratic manifold sparse regression achieve orders of magnitude higher accuracies than linear methods on data describing transport-dominated dynamics in numerical experiments. 
\end{abstract}

\section{Introduction}
Approximating a field variable over a spatial domain using only a few sparse measurements is a fundamental challenge in computational science and engineering. When the field variable is represented as a data vector, this problem becomes approximating the vector from a small subset of the set of all of its components. 
Empirical interpolation and regression methods address this challenge by first constructing a suitable approximation space from training data and then approximating new, unseen data vectors in this space using only a few sparse samples given by the  components of the data vectors that are available. Prominent examples of such methods include gappy proper orthogonal decomposition \cite{EversonS1995Karhunen-Loeve} and empirical interpolation \cite{BarraultMNP2004Empirical,ChaturantabutS2010Nonlinear}.
However, these approaches are constrained by the limited expressivity of linear approximations in vector spaces, which is related to the Kolmogorov barrier \cite{Peherstorfer2022Breaking}. 
To overcome the limitations of linear approximation spaces, we introduce quadratic manifold sparse regression (QMSR) that seeks nonlinear approximations on quadratic manifolds \cite{JainTRR2017quadratic,RutzmoserRTJ2017Generalization,BarnettF2022Quadratic,SchwerdtnerP2024Greedy,BennerGHP-D2023quadratic,GoyalB2024Generalized} from sparse samples of the data vectors.
The proposed QMSR method first constructs a quadratic manifold from training data and then computes approximations via nonlinear projections that correspond to solutions of sparse regression problems. 
Numerical experiments demonstrate that QMSR approximations are accurate even when observing only a small number of components. Furthermore, the experiments show that the increased expressivity of quadratic manifolds is leveraged by QMSR to accurately approximate data sets stemming from transport-dominated and wave-like problems; problems that are out of reach for methods that are based on linear approximation spaces such as empirical interpolation and gappy proper orthogonal decomposition.

Empirical interpolation \cite{BarraultMNP2004Empirical,MadayM2013Generalized,MadayMPY2015Generalized,MadayMT2016Convergence}, its discrete counter part \cite{ChaturantabutS2010Nonlinear,DrmacG2016New,doi:10.1137/140978430,doi:10.1137/15M1015868,DrmacS2018Discrete,Saibaba2020Randomized,ChellappaFB2024Discrete,doi:10.1137/22M1484018}, missing point estimation \cite{AstridWWB2004Missing,AstridWWB2008Missing}, and gappy proper orthogonal decomposition \cite{EversonS1995Karhunen-Loeve} have been extensively studied for constructing data vectors and field variables from sparse samples. There is work on selecting sampling operators \cite{DrmacG2016New,ZimmermannW2016Accelerated,PeherstorferDG2020Stability,BinevCMN2018Greedy}, which is related to sensor placement \cite{ManoharBKB2018Data-Driven,SargsyanBK2015Nonlinear,ManoharBK2017Environment,ArgaudBCGMM2018Sensor}. There is a range of empirical regression techniques that are referred to as hyper-reduction methods \cite{CarlbergFCA2013GNAT,FarhatACC2014Dimensional,FarhatCA2015Structure-preserving}, which focus on preserving energy and other properties of interest in fully discrete finite-element approximations. There is also a line of work that considers structure preservation in discrete empirical interpolation \cite{ChaturantabutBG2016Structure-Preserving,PagliantiniV2023Gradient-preserving}. Missing point estimation \cite{AstridWWB2004Missing,AstridWWB2008Missing} has been applied to reconstruct flow fields in computational fluid dynamics from sparse samples \cite{Willcox2006Unsteady}. 
However, all of these methods seek linear approximations in subspaces, which is affected by the Kolmogorov barrier. 
There is a large body of work that aims to find nonlinear approximations to circumvent the Kolmogorov barrier; we refer to the brief survey \cite{Peherstorfer2022Breaking} for an overview. 
With regards to nonlinear approximations in  empirical interpolation and regression, there is a line of work on adaptive empirical interpolation that adapts the approximation space as new sparse samples are received so that the Kolmogorov barrier can be circumvented \cite{PeherstorferW2015Online,ZimmermannPW2018Geometric,Peherstorfer2020Model,10.1063/5.0169392}; however, the adaptation is meaningful only in settings where there is a time variable or some other variable that monotonically increases such as the iteration counter in an optimization problem; see also \cite{HesthavenPR2022Reduced} for a survey. There are also localized methods \cite{EftangS2012Parameter,AmsallemZF2012Nonlinear,PeherstorferBWB2014Localized} that pre-compute a finite number of approximation spaces and then pick one depending on a criteria for approximating the data vector at hand. 
We build on work that achieves nonlinear approximations via quadratic approximations, which has been used in the context of model reduction first in the works \cite{JainTRR2017quadratic,RutzmoserRTJ2017Generalization}, which have led to a series of works that consider quadratic approximations in non-intrusive settings \cite{GeelenWW2023Operator,GeelenBWW2024Learning,GeelenBW2023Learning,SchwerdtnerP2024Greedy,SHARMA2023116402} and in fully discrete approximations \cite{BarnettF2022Quadratic}.
However, neither of these works aim to approximate data vectors from sparse samples on quadratic manifolds. 

In this work, we derive a sparse encoder function for quadratic manifold approximations that can operate on sparse samples of data vectors rather than on all components.
Building on the sparse encoder function and previous work on the greedy construction of quadratic manifolds \cite{SchwerdtnerP2024Greedy,SchwerdtnerMPBOP2024Online}, we then derive a sparse greedy method for constructing the quadratic manifold so that applying the sparse encoder and subsequently the quadratic decoder function leads to a nonlinear projection onto the quadratic manifold.
We propose to select the sparse sampling points via the QDEIM selection operator \cite{DrmacG2016New} and show that under mild conditions the QMSR approximations exactly recover the points that lie on the quadratic manifold from sparse samples. This is a similar property as the one of empirical interpolation and regression that ensures that vectors in the approximation space are exactly recovered from sparse samples under mild conditions. 
Numerical experiments demonstrate that QMSR approximations achieve errors that are comparable to the errors achieved by reconstructing the full data points on quadratic manifolds.

This manuscript is organized as follows. We first provide preliminaries in Section~\ref{sec:Prelim} and then present the QMSR method in Section~\ref{sec:sparseGQM}. Numerical experiments are discussed in Section~\ref{sec:NumExp}. Conclusions are drawn in Section~\ref{sec:Conc}.

\section{Preliminaries and problem formulation}\label{sec:Prelim}
This section briefly recapitulates dimensionality reduction with quadratic manifolds \cite{JainTRR2017quadratic,GeelenWW2023Operator,BarnettF2022Quadratic} and a greedy method \cite{SchwerdtnerP2024Greedy} for constructing them.

\subsection{Data reduction with encoder and decoder functions}
Consider the high-dimensional data vectors $\fullstatei{1}, \dots, \fullstatei{\nsnapshots} \in \R^{\nfull}$ that are collected in a data matrix $\snapshots = [\fullstatei{1}, \dots, \fullstatei{\nsnapshots}] \in \R^{\nfull \times \nsnapshots}$. Notice that a data vector can correspond to a field variable that is evaluated at $\nfull$ grid points or an unstructured cloud of $\nfull$ points in the spatial domain. 
Now let $\encoder: \R^{\nfull} \to \R^{\nred}$ be an encoder function that maps a high-dimensional data point $\fullstate \in \mathbb{R}^{\nfull}$ onto a reduced point $\redstate = \encoder(\fullstate)$ in $\mathbb{R}^{\nred}$ with $\nred \ll \nfull$. With a decoder function $\decoder: \R^{\nred} \to \R^{\nfull}$, a reduced data point $\redstate$ can be lifted back into $\mathbb{R}^{\nfull}$.
The reconstruction error with respect to the Euclidean norm $\|\cdot\|_2$ is 
    \begin{align}
      \label{eq:reconstruction_error}
      e(\fullstate) = \|\decoder(\encoder(\fullstate))- \fullstate\|_2^2\,.
    \end{align}
The linear encoder and decoder functions for dimension $\nred$ that minimize the average reconstruction error \eqref{eq:reconstruction_error} of the data $\fullstatei{1}, \dots, \fullstatei{\nsnapshots}$ are given by the first $\nred$ left-singular vectors of the data matrix $\snapshots$: Let  $\leftsings \singvals \rightsings^\top=\snapshots$ be the reduced singular value decomposition of $\snapshots$ 
where $\leftsings \in \mathbb{R}^{\nfull \times r_{\text{max}}}$ is the matrix of left singular vectors, $\singvals \in \mathbb{R}^{r_{\text{max}} \times r_{\text{max}}}$ is the diagonal matrix of singular values $\sigma_1 \geq \cdots \geq \sigma_{r_{\text{max}}} > 0$, and $\rightsings \in \mathbb{R}^{k \times r_{\text{max}}}$ is the matrix of right singular vectors.
Let now  $\Vlin = [\bfphi_1, \dots, \bfphi_{\nred}] \in \mathbb{R}^{\nfull \times \nred}$ contain as columns the first $\nred \leq r_{\text{max}}$ left-singular vectors corresponding to the largest singular values of $\snapshots$. 
The linear encoder and decoder functions that minimize the averaged reconstruction error over the data $\fullstatei{1}, \dots, \fullstatei{\nsnapshots}$ are $\encoder_{\Vlin}(\fullstate)=\Vlin^\top \fullstate$ and $\decoder_{\Vlin}(\redstate)=\Vlin \redstate$, respectively.

\subsection{Decoder functions for approximations on quadratic manifolds}
Approximations on quadratic manifolds as used in \cite{GeelenBW2023Learning, GeelenBWW2024Learning, GeelenWW2023Operator, BarnettF2022Quadratic}, that follow the early work of~\cite{JainTRR2017quadratic, RutzmoserRTJ2017Generalization}, can achieve lower reconstruction errors than linear approximations given by linear encoder and decoder functions. For this, the decoder function is augmented with an additive nonlinear correction term  
  \[
  \decoder_{\Vlin, \Vnonlin}(\redstate) = \Vlin \redstate + \Vnonlin \featuremap(\redstate)\,,
  \]
where $\featuremap: \R^{\nred} \to \R^{\nredmod}$ is a nonlinear feature map onto a $\nredmod$-dimensional feature space $\mathbb{R}^{\nredmod}$, $\Vnonlin \in \mathbb{R}^{\nfull \times \nredmod}$ is a weighting matrix, and $\Vlin \in \mathbb{R}^{\nfull \times \nred}$ is a basis matrix with orthonormal columns. 
The correction term via the nonlinear feature map $h$ allows the decoder $\decoder_{\Vlin,\Vnonlin}$ to reach the points in the set
\begin{equation}\label{eq:Prelim:ManifoldDef}
      \mathcal{M}_{\nred}(\Vlin, \Vnonlin)=\{ \Vlin \redstate + \Vnonlin \featuremap(\redstate) \,|\, \redstate \in \R^{\nred} \} \subset \R^{\nfull}\,,
    \end{equation}
    which can contain points outside of the $\nred$-dimensional subspace spanned by the columns of $\Vlin$.
In this work, we only consider quadratic feature maps 
    \begin{align}
      \begin{split}
      \label{eq:condensed_kronecker}
      \featuremap&: \R^{\nred} \to \R^{\nred(\nred+1)/2},\\ \redstate &\mapsto \begin{bmatrix}
        \redveci{1} \redveci{1} & \redveci{1} \redveci{2} & \cdots & \redveci{1} \redveci{\nred} & \redveci{2} \redveci{2} &\cdots& \redveci{r} \redveci{r}
      \end{bmatrix}^\top\,,
      \end{split}
    \end{align}
with $\redstate = [\redveci{1}, \dots, \redveci{r}]^T$. 
We follow the convention in \cite{JainTRR2017quadratic, GeelenBW2023Learning, GeelenBWW2024Learning, GeelenWW2023Operator, BarnettF2022Quadratic,SchwerdtnerP2024Greedy,SchwerdtnerMPBOP2024Online} and refer to sets $\mathcal{M}_r(\Vlin, \Vnonlin)$ induced by a quadratic $h$ given in \eqref{eq:condensed_kronecker} as quadratic manifold, even though the set $\mathcal{M}_r(\Vlin, \Vnonlin)$ is not necessarily a manifold.  
We remark that other feature maps than quadratic ones can be used such as maps that are trained on data as in \cite{BarnettFM2023Neural-network-augmented}.

\subsection{Linear and nonlinear encoder functions for quadratic manifold approximations}\label{sec:Prelim:LinNonlinDecoder}
Additionally to the decoder function $g_{\Vlin, \Vnonlin}$, we need to devise a strategy  to approximate or encode a high-dimensional data point $\fullstate \in \mathbb{R}^{\nfull}$ on the manifold $\mathcal{M}_r(\Vlin, \Vnonlin)$. 
An approximation $\hat{\fullstate}^* \in \mathcal{M}_{\nred}$ with minimal error is given as a solution of the optimization problem
\begin{equation}\label{eq:Prelim:QMasRegOnM}
\hat{\fullstate}^* \in \operatorname*{arg\,min}_{\hat{\fullstate} \in \mathcal{M}_{\nred}} \|\hat{\fullstate} - \fullstate\|_2^2\,.
\end{equation}
For every solution $\hat{\fullstate}^*$ of \eqref{eq:Prelim:QMasRegOnM}, there exists an $\redstate^* \in \mathbb{R}^{\nred}$ that solves 
\begin{equation}\label{eq:Prelim:QMasReg}
\redstate^* \in \operatorname*{arg\,min}_{\redstate \in \mathbb{R}^{\nred}} \|g_{\Vlin,\Vnonlin}(\redstate) - \fullstate\|_2^2
\end{equation} and vice versa, which follows from the definition of the set $\mathcal{M}_r(\Vlin,\Vnonlin)$ given in \eqref{eq:Prelim:ManifoldDef}, since every point $\hat{\fullstate} \in \mathcal{M}_r$ can be written as $g_{\Vlin, \Vnonlin}(\redstate)$ with an appropriate $\redstate$. 
Even though the problem \eqref{eq:Prelim:QMasRegOnM} and the equivalent problem \eqref{eq:Prelim:QMasReg} do not necessarily have a unique solution, we interpret $\hat{\fullstate}^* = g_{\Vlin,\Vnonlin}(\redstate^*)$ as a nonlinear projection of $\fullstate$ onto the set $\mathcal{M}_{\nred}$ with respect to metric induced by  the norm $\|\cdot\|_2$.

Deriving a nonlinear projection with respect to the $\|\cdot\|_2$ norm can be computationally demanding because the nonlinear regression problem \eqref{eq:Prelim:QMasRegOnM} or equivalently \eqref{eq:Prelim:QMasReg} has to be solved.
Instead, a common choice \cite{GeelenBWW2024Learning,BarnettF2022Quadratic,SchwerdtnerP2024Greedy} is
linearizing $\decoder_{\Vlin,\Vnonlin}$ about the point $\boldsymbol 0_r \in \R^{\nred}$ for finding the encoding, which leads to the optimization problem
\begin{equation}\label{eq:Prelim:LinEncoderLSQ}
\redstate^* = \operatorname*{arg\,\min}_{\redstate \in \mathbb{R}^{\nred}} \|\bar{g}_{\Vlin,\Vnonlin}(\redstate) - \fullstate\|_2^2\,,
\end{equation}
with the linearization $\bar{\decoder}_{\Vlin,\Vnonlin}(\redstate) = \Vlin\redstate$ 
and the unique solution $\redstate^* = \Vlin^{\top}\fullstate$.
The least-squares problem \eqref{eq:Prelim:LinEncoderLSQ} motivates using the linear encoder function given by $f_{\Vlin}(\fullstate) = \Vlin^{\top}\fullstate$, where $\Vlin$ is the same matrix with orthonormal columns that is used in the decoder $g_{\Vlin, \Vnonlin}$. 
The linear encoder function $f_{\Vlin}$  is used in \cite{GeelenBWW2024Learning,BarnettF2022Quadratic,SchwerdtnerP2024Greedy} and we will use it in the following as well.
\subsection{Greedy method for constructing quadratic manifolds}
Given a data matrix $\snapshots$ and a feature map $h$, there exist different methods for constructing a basis $\Vlin$ with orthonormal columns and a weight matrix $\Vnonlin$, which induces the linear encoder $f(\fullstate) = \Vlin^{\top}\fullstate$ and the quadratic decoder function $g(\redstate) = \Vlin\redstate + \Vnonlin h(\redstate)$ corresponding to the manifold \eqref{eq:Prelim:ManifoldDef}. In~\cite{GeelenWW2023Operator}, the authors propose to select $\Vlin$ as the first $\nred$ left-singular vectors of $\snapshots$, while~\cite{GeelenBW2023Learning} introduce an elaborate alternating minimization approach. In this work, we build on~\cite{SchwerdtnerP2024Greedy}, which introduces a greedy method for choosing a basis from the set of the leading $\nconsider > \nred$ left-singular vectors, which is empirically shown to lead to a higher accuracy than just selecting the first $\nred$ and which is computationally less taxing than the alternating minimization of~\cite{GeelenBW2023Learning}.
The greedy method in~\cite{SchwerdtnerP2024Greedy} selects $\nred$ left-singular vectors $\leftsing{j_1}, \dots, \leftsing{j_{\nred}}$ with indices $j_1, j_2, \dots, j_{\nred} \in \{1, \dots, \nconsider\}$  from the first $\nconsider \gg \nred$ left-singular vectors $\leftsing{1}, \dots, \leftsing{\nconsider}$. The selected left-singular vectors form the columns of the matrix $\Vlin = [\leftsing{j_1}, \dots, \leftsing{j_{\nred}}]$.
Then, the weight matrix $\Vnonlin$ is fitted by solving the linear least-squares problem
    \begin{align}
      \label{eq:qm_lstsq_problem}
      \min\limits_{\Vnonlin \in \R^{\nfull \times \nredmod}} \frobsq{(\bfI - \Vlin \Vlin^\top) \snapshots + \Vnonlin \featuremap(\Vlin^\top \snapshots)} + \gamma \frobsq{\Vnonlin},
    \end{align}
    where $\gamma > 0$ is a regularization parameter and $\frob{\cdot}$ denotes the Frobenius norm of its matrix argument.

\section{Sparse regression on quadratic manifolds}\label{sec:sparseGQM}
We introduce the quadratic manifold sparse regression (QMSR) method for constructing approximations $\hat{\fullstate} \in \mathbb{R}^{\nfull}$ from $\nrows \ll \nfull$ components $s_{i_1}, \dots, s_{i_{\nrows}}$ with indices $i_1, \dots, i_{\nrows}$ of a data vector $\fullstate = [s_1, \dots, s_{\nfull}]^{\top} \in \mathbb{R}^{\nfull}$.
We first formulate a nonlinear sparse regression problem on quadratic manifolds, which leads to a sparse linear encoder function. 
We then propose a sparse greedy method for training the sparse linear encoder and the quadratic decoder function of a quadratic manifold. For the training we have available all components of training data points $\fullstatei{1}, \dots, \fullstatei{k}$. 
The sparse encoder function and the decoder function corresponding to the quadratic manifold can then be used to derive approximations $\hat{\fullstate}$ from sparse samples $s_{i_1}, \dots, s_{i_{\nrows}}$ of a data point $\fullstate$.

\subsection{Approximating data points on quadratic manifolds from sparse samples}\label{sec:QMSR:SparseEncoder}
Let $\subsampler: \R^{\nfull} \to \R^{\nrows}$ denote a sampling operator that selects $\nrows$ components corresponding to the indices $i_1, \dots, i_{\nrows} \in \{1, \dots, \nfull\}$ out of the $\nfull$ components of a point $\fullstate = [s_1, \dots, s_{\nfull}]^{\top} \in \mathbb{R}^{\nfull}$,
\[
\subsampler \fullstate = \begin{bmatrix}
s_{i_1}\\
\vdots\\
s_{i_{\nrows}}
\end{bmatrix} \in \mathbb{R}^{\nrows}.
\]
We refer to $s_{i_1}, \dots, s_{i_{\nrows}}$ as sparse samples of $\fullstate$ because we are typically interested in cases with $\nrows \ll \nfull$. Consider now a quadratic decoder function $\decoder_{\Vlin,\Vnonlin}(\redstate) = \Vlin\redstate + \Vnonlin h(\redstate)$. 
Analogously to the case where all components of the data point $\fullstate$ are available as discussed in Section~\ref{sec:Prelim:LinNonlinDecoder}, we can find an encoding $\redstate^*$ with the lowest error via the nonlinear regression problem
\begin{equation}\label{eq:ERQM:OptiProb}
\redstate^* \in \operatorname*{arg\,min}_{\redstate \in \mathbb{R}^{\nred}} \|\subsampler (\decoder_{\Vlin,\Vnonlin}(\redstate) - \fullstate)\|_2^2\,,
\end{equation}
which uses only the available components $\subsampler \fullstate$ to find an $\redstate^*$.
To avoid having to solve a nonlinear least-squares problem, we mimic the approach of Section~\ref{sec:Prelim:LinNonlinDecoder} and use the linearized decoder function $\bar{g}_{\Vlin, \Vnonlin}$ to obtain the optimization problem
\begin{equation}\label{eq:QMSR:LinEncodeLSQ}
\redstate^* = \operatorname*{arg\,min}_{\redstate \in \mathbb{R}^{\nred}} \|\subsampler (\bar{g}_{\Vlin,\Vnonlin}(\redstate) - \fullstate)\|_2^2\,,
\end{equation}
which has the unique solution $\redstate^* = (\subsampler \Vlin)^{+}\subsampler \fullstate$ if $\subsampler \Vlin$ has full rank, where $\Vlin^{+}$ denotes the Moore-Penrose pseudo-inverse of $\Vlin$.
Analogously to the case when all components of $\fullstate$ are available, the optimization problem \eqref{eq:QMSR:LinEncodeLSQ} motivates the sparse linear encoder function
\[
\subencoder: \mathbb{R}^{\nrows} \to \mathbb{R}^{\nred}\,,\qquad \subencoder(\subsampler \fullstate) = (\subsampler \Vlin)^{+}\subsampler \fullstate\,.
\]
Critically, the sparse encoder function $\subencoder$ only requires the $\nrows$ components $\subsampler\fullstate$ of $\fullstate$ corresponding to the sampling operator $\subsampler$.
Even though we will focus on the encoder function $\subencoder$, we  will also compare to the encoding obtained by numerically solving \eqref{eq:ERQM:OptiProb} in the computational experiments in Section~\ref{sec:NumExp}. 

In summary, given sparse samples $s_{i_1}, \dots, s_{i_{\nrows}}$ of a point $\fullstate \in \mathbb{R}^{\nfull}$, the QMSR approximation is $\hat{\fullstate} = g_{\Vlin,\Vnonlin}(f_{\Vlin,\subsampler}(\subsampler \fullstate))$. In the following sections, we discuss a sparse greedy method to construct $\Vlin$ and $\Vnonlin$ as well as the sampling operator $\subsampler$. 

\subsection{Greedy construction of linear encoder and quadratic decoder for sparse data}\label{sec:QMSR:SparseGreedy}
We now propose a sparse greedy method for constructing $\Vlin$ and $\Vnonlin$ for quadratic manifolds with sparse encoder $\subencoder$ and quadratic decoder $g_{\Vlin,\Vnonlin}$ that mimics the greedy method introduced in \cite{SchwerdtnerP2024Greedy} for full data. 
Given are a sampling operator $\subsampler$ and training data $\fullstatei{1}, \dots, \fullstatei{k}$, of which all components are available. 

Let $f_{[\Vlin, \testvec], \subsampler}$ be the sparse encoder function corresponding to the matrix $[\Vlin, \testvec]$ that contains as  columns the columns of $\Vlin$ and as last column the vector $\testvec$. The columns of $[\Vlin, \testvec]$ are orthonormal. We define the objective function
    \begin{align}
      \label{eq:ourmethod:objective}
\columnobjective(\testvec,\Vlin,\Vnonlin)=\frobsq{
        [\Vlin, \testvec]f_{[\Vlin,\testvec], \subsampler}(\subsampler \snapshots) + \Vnonlin \featuremap\left(f_{[\Vlin,\testvec], \subsampler}(\subsampler \snapshots)\right)-\snapshots
      } + \gamma \frobsq{\Vnonlin},
    \end{align}
  where we overload the notation of the encoder function $f_{[\Vlin,\testvec], \subsampler}$ to operate column-wise on a matrix rather than only on a vector. 
The term $\gamma\frobsq{\Vnonlin}$ acts as a regularizer on the coefficient matrix $\Vnonlin$, with regularization parameter $\gamma > 0$. 
Notice that the objective function $\columnobjective$ is analogous to the objective function used in the greedy method introduced in \cite{SchwerdtnerP2024Greedy} except that our objective  $\columnobjective$ uses the sparse encoder function. 

Recall that $\leftsing{1}, \dots \leftsing{\nconsider}$ denotes the first $\nconsider$ left-singular vector of the data matrix $\snapshots$ ordered descending with respect to the singular values $\singval{1} \ge \singval{2} \ge \dots \ge \singval{\nconsider}$. 
We greedily select $\nred$ left-singular vectors $\leftsing{j_1}, \dots, \leftsing{j_\nred}$ from the first $\nconsider$ left-singular vectors  with the objective~\eqref{eq:ourmethod:objective}, i.e.,\ in each greedy iteration we determine a minimizer
    \begin{align}
      \label{eq:ourmethod:minimizationprob}
      \min\limits_{j_i=1, \dots,\nconsider} \min\limits_{\Vnonlin\in\R^{\nfull \times \nredmod}} \columnobjective(\leftsing{j_i}, \Vlin_{i-1}, \Vnonlin),
    \end{align}
    where $\Vlin_{i-1} = [\leftsing{j_1}, \dots, \leftsing{j_{i-1}}]$ denotes the matrix constructed during the previous $i - 1$ iterations. 
We initialize $\Vlin_0 \in \R^{\nfull \times 0}$ as an empty matrix.
When $\nred$ indices $j_1, \dots, j_{\nred}$ have been selected, we obtain $\Vlin = [\leftsing{j_1}, \dots, \leftsing{j_{\nred}}] \in \mathbb{R}^{\nfull \times \nred}$ and thus the sparse encoder function $\subencoder$. We then  compute $\Vnonlin$ as the minimizer of
\begin{align}
  \label{eq:ourmethod:finalcoeffcomput}
\min\limits_{\Vnonlin\in\R^{\nfull \times \nredmod}}\frobsq{
    \Vlin\subencoder(\subsampler\snapshots) + \Vnonlin \featuremap\left(\subencoder( \subsampler \snapshots)\right)-\snapshots
  } + \gamma \frobsq{\Vnonlin},
\end{align}
to obtain the quadratic decoder $\decoder_{\Vlin,\Vnonlin}(\redstate)=\Vlin \redstate + \Vnonlin \featuremap(\redstate)$.

\begin{remark}We note it can be avoided to evaluate the objective function~\eqref{eq:ourmethod:objective} when determining the column indices during the greedy iterations. In fact, in our implementation, during the greedy iterations, we determine a minimizer
\begin{align}
\label{eq:ourmethod:minimization_prob_cheap}
  \min\limits_{j_i=1, \dots,\nconsider} \min\limits_{\widetilde{\Vnonlin}\in\R^{\nredmax \times \nredmod}} \columnobjectivecheap(\leftsing{j_i}, \Vlin_{i-1}, \widetilde{\Vnonlin}),
\end{align}
where
\begin{align*}
\columnobjectivecheap(\leftsing{j_i}, [\leftsing{j_1}, \dots, \leftsing{j_{i-1}}], \widetilde{\Vnonlin})=\hspace{7.5cm}\\
\frobsq{
    \lifter(\leftsings, j_1, \dots, j_{i-1}, j_i)  + \widetilde{\Vnonlin} \featuremap\left(f_{[\leftsing{j_1}, \dots, \leftsing{j_{i-1}}, \leftsing{j_i}], \subsampler}(\subsampler \snapshots)\right)-\singvals \rightsings^\top
  } + \gamma \frobsq{\widetilde{\Vnonlin}}.
\end{align*}
The $j_\ell$-th row of the matrix $\lifter(\leftsings, j_1, \dots, j_i) \in \R^{\nredmax \times \nsnapshots}$
contains the $\ell$-th row of $f_{[\leftsing{j_1}, \dots, \leftsing{j_{i-1}}, \leftsing{j_i}], \subsampler}(\subsampler \snapshots)$ for $\ell \in \{1,\dots,i\}$ and all other rows of $\lifter(\leftsings, j_1, \dots, j_i)$ are zero. Note that the minimization problems~\eqref{eq:ourmethod:minimizationprob} and~\eqref{eq:ourmethod:minimization_prob_cheap} have the same solutions, since $\frob{\mathbf{A}\mathbf{Z}}=\frob{\mathbf{A}}$, for any $\mathbf{A}\in \R^{n \times m}$ and $\mathbf{Z} \in \R^{m \times p}$, when $\mathbf{Z}$ has orthonormal columns and $p\ge m$. This allows us to solve the least-squares problem with fewer unknowns given by $\columnobjectivecheap$ during the greedy iterations and only solve~\eqref{eq:ourmethod:finalcoeffcomput} once when all columns of $\Vlin$ are determined.
\end{remark}

\subsection{Properties of sparse regression on quadratic manifolds}
We now show two properties of approximations obtained with the QMSR method. We first show that the QMSR approximation $\hat{\fullstate}$ of a point $\fullstate$ is optimal with respect to a least-squares regression objective, which motivates the term ``regression'' in QMSR. We start by proving that applying QMSR is an idempotent map. 

\begin{proposition}
\label{prop:idempotent}
Consider a basis matrix $\Vlin \in \mathbb{R}^{\nfull \times \nred}$ and selection operator $\subsampler: \mathbb{R}^{\nfull} \to \mathbb{R}^{\nrows}$. For any $\fullstate \in \mathbb{R}^{\nfull}$, the composition $\recmap: \R^{\nfull} \to \R^{\nfull}, \fullstate \mapsto \decoder(\subencoder(\fullstate))$ is an idempotent map onto $\mathcal{M}_r(\Vlin, \Vnonlin)$, if $\Vnonlin$ is constructed as a minimizer of~\eqref{eq:ourmethod:finalcoeffcomput}.
\end{proposition}
\begin{proof}
Since we denote $\redstate=\subencoder(\fullstate)=\pvp \subsampler \fullstate$, we have that $\recmap(\fullstate)=\Vlin \redstate + \Vnonlin\featuremap(\redstate)$. Therefore,
\begin{align*}
    \subencoder(\recmap(\fullstate))&=\pvp \subsampler(\Vlin \redstate+\Vnonlin\featuremap(\redstate))=\pvp \subsampler\Vlin \redstate + \pvp \subsampler \Vnonlin \featuremap(\redstate)\\
    &=\underbrace{\pvp \subsampler \Vlin \pvp}_{\pvp} \subsampler \fullstate + \pvp \subsampler \Vnonlin \featuremap(\redstate)\\
    &=\redstate + \pvp \subsampler \Vnonlin \featuremap(\redstate).
\end{align*}
Next we show that $\pvp \subsampler \Vnonlin \featuremap(\redstate)=\boldsymbol{0}\in\R^{\nred}$. For this, note that since we select $\Vnonlin$ as a minimizer of~\eqref{eq:ourmethod:finalcoeffcomput}, $\Vnonlin$ can be written as
\begin{align}
\label{eq:Vnonlinform}
\Vnonlin=(\Vlin\pvp \subsampler-\bfI)\snapshots\bfK,
\end{align}
where 
\begin{align*}
\bfK=(\featuremap(\redstate)^\top + \gamma \bfI)\left((\featuremap(\redstate)^\top + \gamma \bfI)^\top (\featuremap(\redstate)^\top + \gamma \bfI)\right)^{-1}.
\end{align*} Thus, 
\begin{align}
\label{eq:PVpPW_is_zero}
\begin{split}
\pvp \subsampler \Vnonlin&=(\pvp \subsampler\Vlin \pvp \subsampler-\pvp\subsampler)\snapshots\bfK\\&=(\pvp\subsampler-\pvp\subsampler)\snapshots\bfK=\boldsymbol{0}.
\end{split}
\end{align}
It follows that $\subencoder(\recmap(\fullstate))=\subencoder(\fullstate)$ and thus \begin{align*}
\recmap(\fullstate)=\decoder(\subencoder(\fullstate))=\decoder(\subencoder(\recmap(\fullstate)))=\recmap(\recmap(\fullstate)).
\end{align*}
\end{proof}

Building on Proposition~\ref{prop:idempotent}, we determine the optimization problem that is solved by our QMSR approximation. Notice that we now require that $\subsampler \Vlin$ has full rank and that $\nrows\ge \nred$.

\begin{proposition}
If $\subsampler\Vlin$ has full rank, $\nrows \ge \nred$, and $\Vnonlin$ is obtained with the sparse greedy procedure so that $\Vnonlin$ solves \eqref{eq:ourmethod:finalcoeffcomput}, then, for all $\fullstate \in \mathbb{R}^{\nfull}$, the QMSR approximation
\[
\hat{\fullstate} = \decoder_{\Vlin,\Vnonlin}(f_{\Vlin,\subsampler}(\fullstate))
\]
is the unique solution of
\begin{align}\label{eq:QMSRProp:ProofUniqueSol}
\operatorname*{arg\,min}_{\hat{\fullstate} \in \mathcal{M}_r} \|(\subsampler\Vlin)^{+}\subsampler(\hat{\fullstate} - \fullstate)\|_2^2\,.
\end{align}
\end{proposition}
\begin{proof}
We can rewrite~\eqref{eq:QMSRProp:ProofUniqueSol} as
\begin{align*}
\label{eq:QMSROptProb}
\operatorname*{arg\,min}_{\redstate \in \R^{\nred}} \|(\subsampler\Vlin)^{+}\subsampler(\Vlin \redstate + \Vnonlin \featuremap(\redstate) - \fullstate)\|_2^2=
\operatorname*{arg\,min}_{\redstate \in \R^{\nred}} \|(\subsampler\Vlin)^{+}\subsampler(\Vlin \redstate - \fullstate)\|_2^2,
\end{align*}
where we used~\eqref{eq:PVpPW_is_zero}. This is a linear least-squares problem, which has the unique solution $\redstate=\pvp \subsampler \fullstate$, when $\subsampler \Vlin$ has full column rank.
\end{proof}

We now show that if we have given a point $\fullstate \in \mathcal{M}_r$ on the quadratic manifold, then the QMSR approximation $\hat{\fullstate}$ is the point $\fullstate$ if $\subsampler\Vlin$ has full rank. This property is analogous to the property of empirical interpolation \cite{BarraultMNP2004Empirical,ChaturantabutS2010Nonlinear} that a point that is in the linear approximation space is mapped to itself, under certain conditions.
We now show that the analogous property holds for QMSR. 

\begin{corollary}
Denote the QMSR approximation by $\hat{\fullstate} = g_{\Vlin,\Vnonlin}(f_{\Vlin,\subsampler}(\subsampler\fullstate))$. If $\nrows \ge \nred$ and $\subsampler \Vlin$ has full column rank, then, for the QMSR approximation of a point $\fullstate \in \mathcal{M}_r$ on the manifold, we have that $\hat{\fullstate}=\fullstate$.
\end{corollary}
\begin{proof}
We can write any $\fullstate \in \mathcal{M}_r$ as $\Vlin \redstate + \Vnonlin\featuremap(\redstate)$ for an appropriate $\redstate \in \R^{\nred}$. Then we have that
\begin{align*}
f_{\Vlin,\subsampler}(\fullstate)=\pvp \subsampler (\Vlin \redstate+\Vnonlin\featuremap(\redstate))=\pvp \subsampler \Vlin \redstate=\redstate,
\end{align*}
where the second equality is again a consequence of the construction of $\Vnonlin$, that is also used in the proof of Proposition~\ref{prop:idempotent} and the third equality holds when $\pvp$ is the left inverse of $\subsampler \Vlin$, which is the case when $\subsampler \Vlin$ has full column rank. Since for any $\fullstate = \Vlin \redstate + \Vnonlin\featuremap(\redstate) \in \mathcal{M}_r$ it holds that $f_{\Vlin,\subsampler}(\fullstate)=\redstate$, we have that $\hat{\fullstate}=\fullstate$.
\end{proof}

\begin{algorithm}[t]
  \caption{Empirical sparse regression on quadratic manifolds}
  \label{algo:main}
  \begin{algorithmic}[1]
    \Procedure{QMSR}{$\snapshots, \nrows, \nred, \nconsider, \gamma$}
    \State{Compute the first $\nconsider$ singular vectors $\leftsing{1},\dots,\leftsing{\nconsider}$ of the data matrix $\snapshots$.} 
    \State{Apply \texttt{QDEIM} to $\Phi_{\nrows} = [\leftsing{1}, \dots, \leftsing{\nrows}]$ to obtain sampling operator $\subsampler$.} 
    \State{Set $\mathcal{I}_0=\{\}, \Vlin_0 = []$}
    \For{$i = 1,\dots, r$}
    \State{Compute $\leftsing{j_i}$ that minimizes~\eqref{eq:ourmethod:minimizationprob} over $\leftsing{j^{1}}, \dots, \leftsing{j^m}$ for $\Vlin_{i-1}$ and $\subsampler$.}
    \State{Set $\Vlin_i = [\leftsing{1}, \dots, \leftsing{j_i}]$}
    \EndFor
    \State{Set $\Vlin= \Vlin_r$} 
    \State{Fit $\Vnonlin$ via the least-squares problem \eqref{eq:ourmethod:finalcoeffcomput} with regularization parameter $\gamma$.}
    \State{Return $\Vlin, \Vnonlin, \subsampler$.}
    \EndProcedure
  \end{algorithmic}
\end{algorithm}

\subsection{Constructing a sampling operator}
The proposed sparse greedy method and QMSR are applicable with any subsampling operator $\subsampler$. We propose here one approach for constructing $\subsampler$ that leverages that the greedy method selects the columns of $\Vlin$ from the first $\nconsider$ left-singular vectors. Because we already need the operator $\subsampler$ to select the columns of $\Vlin$, we select $\subsampler$ based on the first $\nrows$ left-singular vectors of the data matrix $\snapshots$. In particular, we aim to select $\subsampler$ such that the data $\snapshots$ can be well approximated in the space spanned by the first $\nrows$ left-singular vectors, which we achieve by applying QDEIM \cite{DrmacG2016New} to the matrix $\leftsings_{\nrows} = [\leftsing{1}, \dots, \leftsing{\nrows}]$ given by the first $\nrows$ left-singular vectors. Notice that any other method for selecting samplings points in empirical interpolation can be applied \cite{BarraultMNP2004Empirical, ChaturantabutS2010Nonlinear,PeherstorferDG2020Stability}. 
Even though the decoder function $g_{\Vlin,\Vnonlin}$ is quadratic in our case, the sparse encoder function $\subencoder$ is still linear. 
By using $\nrows > \nred$ singular vectors in $\leftsings_{\nrows}$, we take into account that $\Vlin$ is constructed by the sparse greedy method with later left-singular vectors. 

\subsection{Algorithms for QMSR}
We summarize the procedure to construct the sampling operator $\subsampler$, the basis matrix $\Vlin$, and the weight matrix $\Vnonlin$ for QSMR in Algorithm~\ref{algo:main}. The inputs to the algorithm are the data matrix $\snapshots$, the number of sparse samples $\nrows$, the reduced dimension $\nred$, the number $\nconsider$ of left-singular vectors to consider  in the sparse greedy method, and the regularization parameter $\gamma$. First, the left-singular vectors of the data matrix $\snapshots$ are constructed and the sampling operator $\subsampler$ is constructed by applying QDEIM to the matrix with the first $\nrows$ left-singular vectors as columns. Then the sparse greedy method described in Section~\ref{sec:QMSR:SparseGreedy} is applied. It selects the $\nred$ left-singular vectors from the first $\nconsider$ left-singular vectors of $\snapshots$ to assemble the basis matrix $\Vlin$. The weight matrix $\Vnonlin$ is then fitted via the least-squares problem \eqref{eq:ourmethod:finalcoeffcomput} with regularization parameter $\gamma$. The procedure returns the basis matrix $\Vlin$, the weight matrix $\Vnonlin$, and the sampling operator $\subsampler$.

\section{Numerical examples}\label{sec:NumExp}
We demonstrate QSMR on three different data sets that are generated by simulating collisionless charged particles, interacting pulse signals, and rotating detonation waves, respectively. 
    In all experiments, we set the regularization parameter in the optimization problem~\eqref{eq:ourmethod:finalcoeffcomput} to $\gamma=10^{-8}$. Code is available on GitHub \url{https://github.com/Algopaul/qmsr\_demo.git}.

In all experiments, we report the relative error of approximating test data points that were not used during training. The relative error is computed as
        \begin{equation}\label{eq:NumExp:RelErr}
          E_{\mathrm{rel}}(\hat{\snapshots}^{(\mathrm{test})}) = \frac{1}{\frob{\snapshots^{(\mathrm{test})}}}\frob{\hat{\snapshots}^{(\text{test})} - \snapshots^{(\mathrm{test})}},
        \end{equation}
where $\snapshots^{(\text{test})}$ is the test data matrix and $\hat{\snapshots}^{(\text{test})}$ the approximation.

\subsection{Collisionless charged particles}
We consider density fields governed by the Vlasov equation with a fixed potential, which describes the motion of collisionless charged particles under the influence of an electric field; see \cite{Gu-cluGCH2014Arbitrarily,BermanP2024CoLoRA} for details about the setup.

  \subsubsection{Setup}
  We study the particle density, denoted by $u$, that is governed by the Vlasov equation 
    \begin{align}
      \partial_t u(t, x_1, x_2) = -x_2 \partial_{x_1}u(t, x_1, x_2) +\phi(x_1)\partial_{x_2}u(t, x_1, x_2)\,,
    \end{align}
    with the potential
    \[
    \phi(x_1)=0.2+0.2\cos(\pi x_1^4)+0.1\sin(\pi x_1),
    \]
    over $x = [x_1, x_2]^T \in [-1, 1)^2$ and $t \in [0, 5)$. 
   The coordinates $x_1$ and $x_2$ correspond to the position and velocity of the particles, respectively. We impose periodic boundary conditions.
   The initial condition is given by $u_0(x_1, x_2)=\exp(-10^2 (x_1^2 + x_2^2))$.
    To numerically solve the Vlasov equation, we discretize the spatial domain $[-1,1)^2$ with 600 grid points in each dimension and use second-order central finite differences for the differential operators, which leads to a state dimension of $\nfull=360,000$.
    The time interval is discretized with time-step size $\delta t=2\times 10^{-3}$ and a fourth-order explicit Runge-Kutta scheme.
    We obtain the data vectors $\fullstatei{i} \in \mathbb{R}^{\nfull}$ at the time steps $i = 1, 2, 3, \dots$, where each component corresponds to the numerical approximation of the solution field $u(t_i, x_1, x_2)$ at time $t_i = (i-1)\delta t$ and at one of the $600 \times 600$ grid points in the domain $[-1, 1)^2$. 
    We split these $2500$ data vectors into $\nsnapshots = 1250$ training and 1250 test data vectors, where data vectors with even time steps $i = 0, 2, 4, \dots$ are training data vectors and data vectors corresponding to odd time steps $i = 1, 3, 5, \dots$ are test data vectors. 

\begin{figure}
    \centering
    \resizebox{1.0\textwidth}{!}{\input{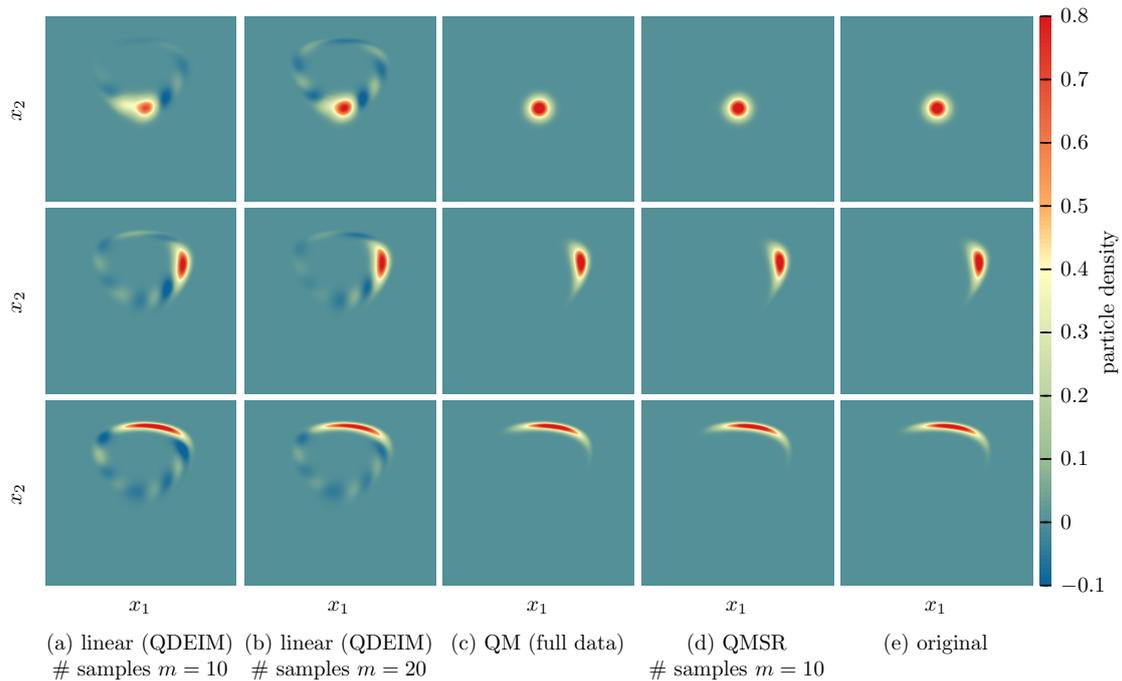}}
    \caption{Charged particles: QMSR approximations are visually indistinguishable from the original data vectors, even though only $m = 10$ sparse samples out of the $n = 360,000$ components are used. QMSR also leverages the expressivity of quadratic manifolds to achieve significantly higher accuracy than the linear approximations obtained with empirical interpolation.} 
    \label{fig:vlasov:reconstruction}
\end{figure}

\begin{figure}
    \centering
    \resizebox{1.0\textwidth}{!}{\input{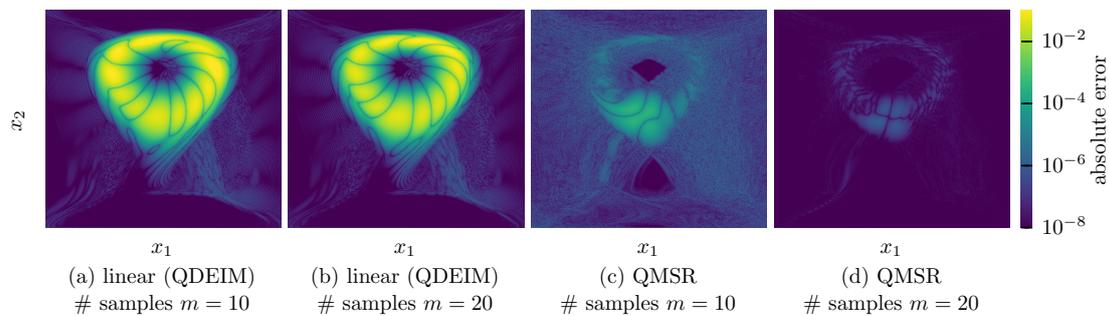}}
    \caption{Charged particles: Doubling the number of sparse samples from $m = 10$ to $m = 20$ significantly increases the accuracy of QMSR approximations in this example. In contrast,  doubling the number of samples has little effect on the linear approximations obtained with empirical interpolation because the linear approximation space is limiting the approximation accuracy.} 
    \label{fig:vlasov:err_reconstruction}
\end{figure}

\begin{figure}
    \centering
    \resizebox{1.0\textwidth}{!}{\input{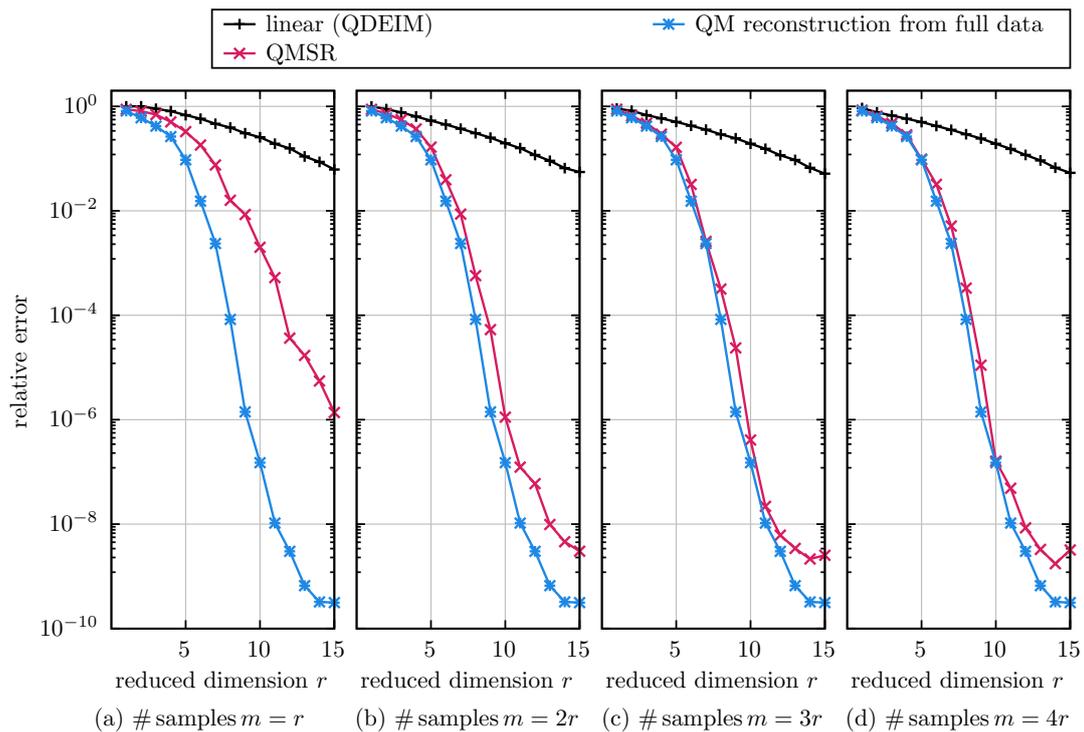}}
    \caption{Charged particles: The QMSR approximations obtained with $m = 2r$ sparse samples achieve a comparable accuracy on test data vectors as reconstructing the full data vectors on the quadratic manifold.}
\label{fig:vlasov:rel_errs}
\end{figure}

\subsubsection{Results}\label{sec:NumExp:Vlasov:Results}
We compare three methods. First, we consider empirical interpolation \cite{BarraultMNP2004Empirical,ChaturantabutS2010Nonlinear} with the QDEIM sampling operator \cite{DrmacG2016New}. The subspace for the linear approximation obtained with empirical interpolation is spanned by the first 
$\nred$ left-singular vectors of the data matrix, which has the training data points as its columns. 
We apply QDEIM to the matrix $\Phi_{\nrows}$ that has as columns the first $\nrows$ left-singular vectors to construct the sampling operator $\subsampler$. Notice that $\nrows \geq \nred$ in our examples.
Second, we consider the reconstruction obtained on a quadratic manifold that is trained with the greedy method proposed in \cite{SchwerdtnerP2024Greedy}. The reconstruction uses the linear encoder $f_{\Vlin}$ that \emph{uses all components} $s_1, \dots, s_{\nfull}$ ($\nrows = \nfull$) of a test data point $\fullstate = [s_1, \dots, s_{\nfull}]^{\top}$ for finding an approximation; see  Section~\ref{sec:Prelim:LinNonlinDecoder}. Third, we consider QMSR that uses only $\nrows \ll \nfull$ sparse samples $\subsampler \fullstate$ of a test data point $\fullstate$ to find an approximation on a quadratic manifold. We obtain the sampling operator $\subsampler$, the basis matrix $\Vlin$, and the weight matrix $\Vnonlin$ as described in Algorithm~\ref{algo:main} from the training data: The sampling operator is constructed from the first $\nrows$ left-singular vectors of the training data matrix and the basis matrix $\Vlin$ and weight matrix $\Vnonlin$ are constructed from the training data points using the sparse greedy method introduced in Section~\ref{sec:QMSR:SparseGreedy}. 

Figure~\ref{fig:vlasov:reconstruction} shows the approximations of the test data vectors corresponding to times $t = 0$, $t  = 1.67$, and $t = 2.5$ obtained with QDEIM with $\nrows = 10$ and $\nrows = 20$ sparse samples, the quadratic manifold reconstruction, and our QMSR with $\nrows = 10$ sparse samples. The reduced dimension is $\nred = 10$ in all cases. The linear approximation given by QDEIM with $\nrows = 10$ and $\nrows = 20$ sparse samples leads to clearly visible errors. In contrast, the approximations on the quadratic manifolds are visually indistinguishable from the original test data points. In particular, notice that the QSMR approximation that receives only $\nrows = 10$ sparse samples---only $\nrows = 10$ components out of a total of $\nfull = 360,000$ components---achieves an approximation that is almost indistinguishable from the reconstruction on a quadratic manifold, which  uses all components from the test data vectors to find the reconstruction on the quadratic manifold. To give an example of which left-singular vectors are selected by our sparse greedy method for QMSR, we provide here the indices for reduced dimension $\nred=10$ and number of sparse sample points $\nrows=10$, 
\[
\Vlin=\begin{bmatrix}\leftsing{1}, \leftsing{3}, \leftsing{18}, \leftsing{2}, \leftsing{140}, \leftsing{9}, \leftsing{59}, \leftsing{182}, \leftsing{42}, \leftsing{118}\end{bmatrix},
\]
whereas the basis of the reconstruction that uses full data is given by
\[
\Vlin=\begin{bmatrix}
\leftsing{1},  \leftsing{3},  \leftsing{5},  \leftsing{6},  \leftsing{2},  \leftsing{10}, \leftsing{16}, \leftsing{11}, \leftsing{17}, \leftsing{24}
\end{bmatrix}.
\]

The point-wise error $|\fullstate-\hat{\fullstate}|$, where $| \cdot |$ is applied component-wise for the test data points corresponding to time $t=2.5$ is plotted in Figure~\ref{fig:vlasov:err_reconstruction}. It can be observed that doubling the number of sample points $\nrows$ leads to only slightly higher accuracy of the QDEIM approximations, whereas doubling the number of sample points from $\nrows = 10$ to $\nrows = 20$ for QMSR achieves significantly higher accuracy of at least one order of magnitude. 

The relative error~\eqref{eq:NumExp:RelErr} is shown in Figure~\ref{fig:vlasov:rel_errs} for reduced dimensions between $\nred = 1$ and $\nred = 15$ and number of sample points $\nrows \in \{\nred, 2\nred, 3\nred, 4\nred\}$. Already with $\nrows = 2\nred$ sparse samples, the QMSR approximation achieves a comparable relative error averaged over all test data points as the reconstruction that uses the full test data points for encoding. In all cases, the linear approximation with QDEIM leads to orders of magnitude lower accuracy. 

\begin{figure}
    \centering
    \resizebox{1.0\textwidth}{!}{\input{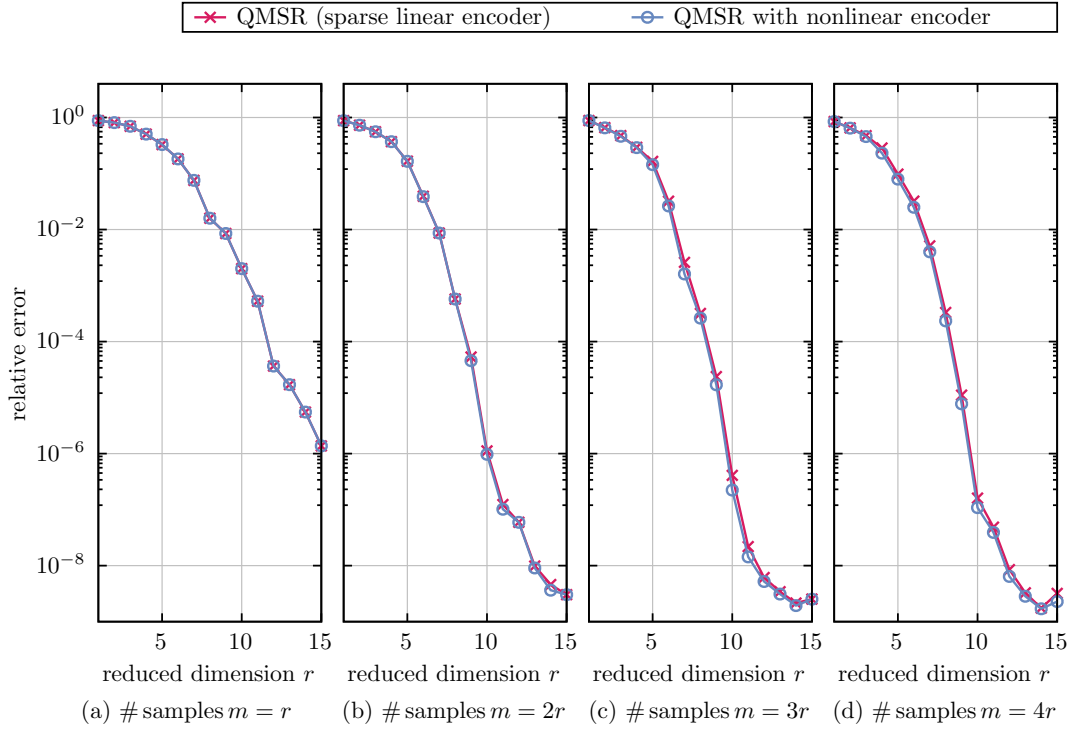}}
    \caption{Charged particles: Computing encoding with the sparse linear encoder of QMSR achieves comparable results as the nonlinear encoding obtained by applying the Gauss-Newton method to nonlinear regression problem \eqref{eq:ERQM:OptiProb}, which indicates that the 
    sparse linear encoder of QMSR is sufficient, at least in this example.}
    \label{fig:vlasov_appendix}
\end{figure}

Recall that QMSR obtains the encoded data point $\redstate$ via a sparse linear encoder function $f_{\Vlin,\subsampler}$. In Figure~\ref{fig:vlasov_appendix} we compare the QMSR approximation that uses the sparse linear encoder $f_{\Vlin,\subsampler}$ to an approximation on the same quadratic manifold that is encoded by solving the sparse nonlinear least-squares problem \eqref{eq:ERQM:OptiProb} with the Gauss-Newton method.
For this, we use at most 20 Gauss-Newton iterations and stop when the difference in the relative errors computed in subsequent iterations is less than $10^{-12}$. To enhance stability and convergence, we add a damping coefficient to the Gauss-Newton iterations. We test damping coefficients $\{10^{-8}, 10^{-4}, \dots, 10^{8}\}$ and report in Figure~\ref{fig:vlasov_appendix} the lowest relative error among all damping coefficients.
The plots in Figure~\ref{fig:vlasov_appendix} show that for reduced dimensions between $\nred = 1$ and $\nred = 15$ and number of sparse samples from $\nrows = \nred$ to $\nrows = 4\nred$, the QMSR approximation with the sparse linear encoder achieves comparable results as when computing the encoding by directly solving the sparse nonlinear least-squares problem with the Gauss-Newton method. This result is in agreement with reconstructions on quadratic manifolds from full data samples, which can be found in \cite{SchwerdtnerP2024Greedy}, and 
provides evidence that the sparse linear encoder function is sufficient, at least in this example. 

\subsection{Hamiltonian interacting pulse signals}
We now consider a traveling pulse signal in a two-dimensional periodic domain, which is governed by the acoustic wave equation; see \cite{SchwerdtnerP2024Greedy} for details about the setup. 

\subsubsection{Setup}
Consider the acoustic wave equation over a two-dimensional spatial domain $[-4, 4)^2$ in Hamiltonian form with periodic boundary conditions,
\begin{align}
\label{eq:acoustic_wave_equation}
  \begin{split}
    \partial_t \rho(t, x) 	&= -\nabla\cdot v(t, x)\,,\\
    \partial_t v(t, x) 		&= -\nabla\rho(t, x)\,, \\
    \rho(0, x) &= \rho_0(x)\,, \\
    v(t, 0) &= 0\,,
  \end{split}
\end{align}
where $\rho: [0, T] \times [-4, 4)^2 \to \R$ denotes the density and $v: [0, T] \times [-4, 4)^2 \to \R^2$ the velocity field. The velocity $v$ consists of two component functions $v_1$ and $v_2$ corresponding to the velocities in the two spatial directions. We set the initial condition to
\begin{equation}
  \rho_0(x_1, x_2) = \exp\left(-(2\pi)^2 \left((x_1 - 2)^2 + (x_2 - 2)^2\right)\right),
\end{equation}
and $v(0, x)=0$. We use a central finite difference scheme with 600 degrees of freedom in each spatial direction, which leads to dimension $\nfull=3 \times 600^2=1\,080\,000$. We collect $1600$ numerical solutions of $\rho$ and $v$ computed with the Runge-Kutta method of 4-th order in the time-interval $[0,8]$ and time-step size $\delta t = 5 \cdot 10^{-3}$.
We obtain 1600 data points $\fullstatei{i}$, at the time steps $i=1,, 2, 3, \dots$, where each component corresponds to the numerical approximation of the concatenated solution fields $\rho(t_i, x_1, x_2),v_1(t_i, x_1, x_2),v_2(t_i, x_1, x_2)$ at time $t_i=(i-1)\delta t$. We use the $\nsnapshots = 800$ data points at even time steps as training data and the ones corresponding to odd time steps as test data points. 

\subsubsection{Results}
Similarly to Section~\ref{sec:NumExp:Vlasov:Results}, we compare approximations obtained with QMSR from $\nrows$ sparse samples to the reconstructions on quadratic manifolds that use the full data points and the linear approximations obtained with QDEIM with $\nrows$ sparse samples. Let us first consider the reduced dimension $\nred = 20$. The basis matrix $\Vlin$ constructed by the sparse greedy method for $\nred=20$ and $\nrows=20$ contains the following left-singular vectors 
  \begin{align*}
  \Vlin=\left[\leftsing{15}, \leftsing{2}, \leftsing{12}, \leftsing{19}, \leftsing{9}, \leftsing{16}, \leftsing{182}, \leftsing{20}, \leftsing{14}, \leftsing{66}, \leftsing{46}, \leftsing{190},\right.\\ \left.\leftsing{192}, \leftsing{29}, \leftsing{76}, \leftsing{62}, \leftsing{114}, \leftsing{13}, \leftsing{151}, \leftsing{91}\right]\,,
  \end{align*}
which again differs from the basis selected for the full reconstruction given by
\begin{align*}
\Vlin=\left[
  \leftsing{17},  \leftsing{11},  \leftsing{13},  \leftsing{20},  \leftsing{22},   \leftsing{2},   \leftsing{10},  \leftsing{18},  \leftsing{35},   \leftsing{5},  \leftsing{16},  \leftsing{25},\right.\\\left.  \leftsing{50},  \leftsing{60},  \leftsing{74},  \leftsing{86},  \leftsing{92}, \leftsing{102}, \leftsing{99},  \leftsing{40}
  \right].
\end{align*}

Figure~\ref{fig:hamwave:reconstruction} shows approximations of the test data points corresponding to times $t \in \{0, 4, 8\}$. 
Our QMSR approach provides approximations that achieve comparable errors as the reconstructions on the quadratic manifold that use the full data vectors, i.e., all $\nfull$ components instead of only $\nrows$ many as QMSR. Linear approximations with QDEIM from $\nrows$ sparse samples lead to orders of magnitude higher errors that approximations on quadratic manifolds in this example. Even doubling the number of sparse samples from $\nrows = 20$ to $\nrows = 40$ leads only to slightly lower errors of the QDEIM approximations, which indicates that indeed the linearity of the approximation is limiting the QDEIM accuracy rather than too few sparse samples. In contrast, as can be observed in the relative point-wise errors shown in Figure~\ref{fig:hamwave:err_reconstruction}, doubling the number of sparse sample points in QMSR leads to a lower error of orders of magnitude. Figure~\ref{fig:hamwave:rel_errs} shows the averaged relative error \eqref{eq:NumExp:RelErr} over all test data points. The QMSR approximations with $\nrows = 2\nred$ sparse samples achieve orders of magnitude lower errors than linear approximations with QDEIM and are close to the accuracy achieved by reconstructions on quadratic manifolds that use all components of the data vectors points for finding encodings. Analogous to the previous example, we show in Figure~\ref{fig:hamwaveGN} that the encoding with the sparse linear encoder function of QMSR achieves comparable accuracy as computing the nonlinear encoding by numerically solving the least-squares problem \eqref{fig:hamwaveGN} with the Gauss-Newton method in this example.

\begin{figure}
    \centering
    \resizebox{1.0\textwidth}{!}{\input{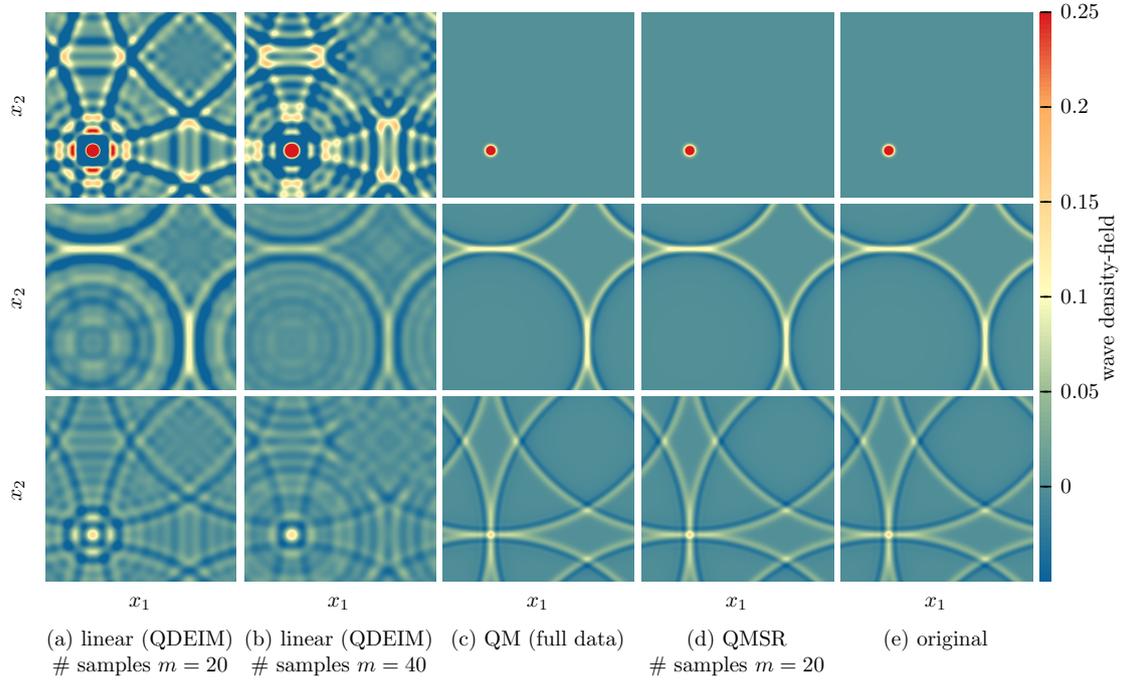}}
    \caption{Hamiltonian wave: The QMSR approximations from $m = 20$ samples out of $n = 1,080,000$ components accurately capture the wave evolutions over time, similarly to the quadratic-manifold reconstructions that use all $n = 1,080,000$ components. Linear approximations with empirical interpolation lead to orders of magnitude higher errors.} 
    \label{fig:hamwave:reconstruction}
\end{figure}

\begin{figure}
    \centering
    \resizebox{1.0\textwidth}{!}{\input{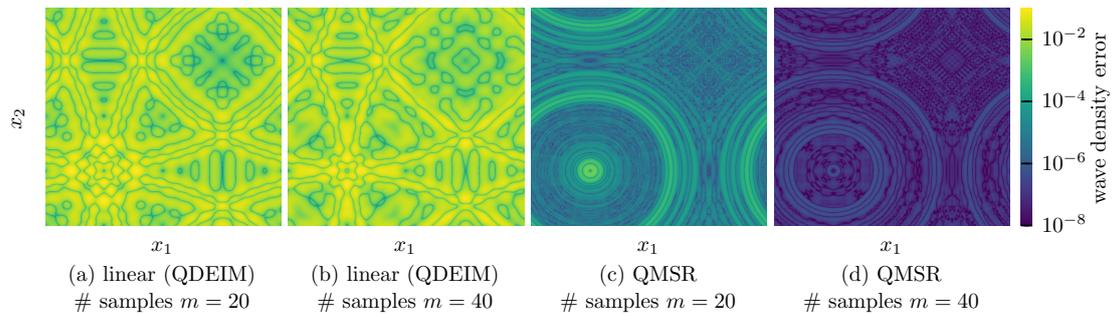}}
    \caption{Hamiltonian wave: Doubling the number $m$ of sampling points decreases the error of QMSR approximations by orders of magnitude. In contrast, doubling the number of sampling points in empirical interpolation has little effect on the error because the linear approximation space is limiting the accuracy in this example.}
    \label{fig:hamwave:err_reconstruction}
\end{figure}

\begin{figure}
    \centering
    \resizebox{1.0\textwidth}{!}{\input{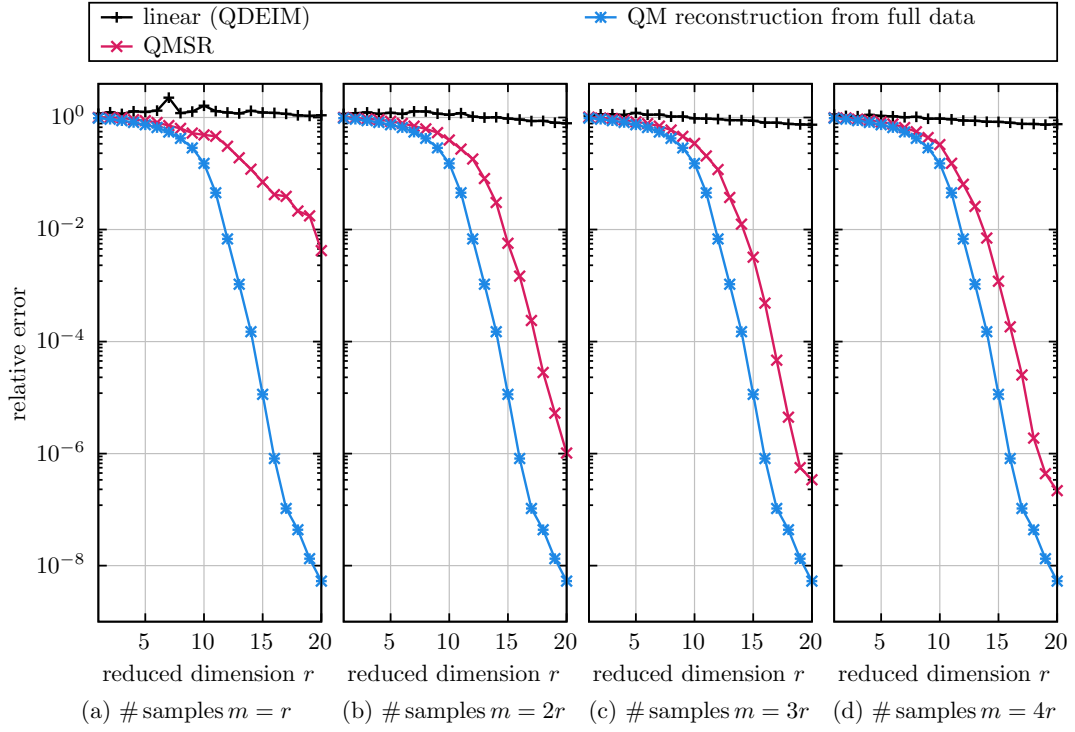}}
    \caption{Hamiltonian wave: The QMSR approximations achieve up to six orders of magnitude higher accuracy than linear approximations while using only up to $m = 80$ sparse samples out of $n = 1,080,000$  components. Notice that empirical interpolation with $m = r$ sparse samples shows instabilities, which is in agreement with \cite{PeherstorferDG2020Stability}.} 
    \label{fig:hamwave:rel_errs}
\end{figure}

\begin{figure}
    \centering
    \resizebox{1.0\textwidth}{!}{\input{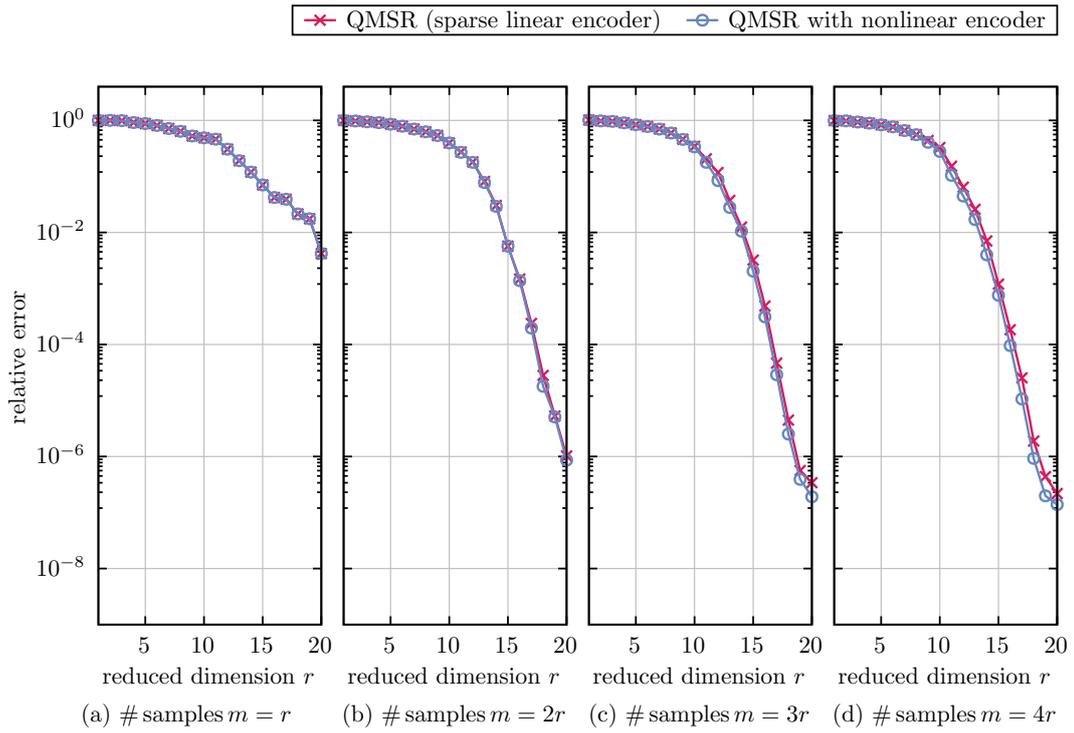}}
    \caption{Hamiltonian wave: The sparse linear encoder of QMSR achieves comparable accuracy as the nonlinear encoding obtain by solving the nonlinear least-squares problem \eqref{eq:ERQM:OptiProb} with the Gauss-Newton method in this example.}
    \label{fig:hamwaveGN}
\end{figure}

\subsection{Rotating detonation waves}
We now consider data that represent rotating detonation waves that arise in rotating detonation engines \cite{AnandG2019Rotating, RamanPG2023Nonidealities}. We use the data available from~\cite{CamachoH2024Investigations}.

\subsubsection{Setup}
The dimension of the data vectors is $\nfull=1,767,500$ in this example. A data vector represents the  pressure, velocities in two spatial directions, temperature, reactant concentrations and densities at a grid of $1414 \times 125$ points in the spatial domain.
We consider a data set that combines trajectories over time obtained for five different inlet velocities $v_{\text{in}} \in \{100\,\mathsf{m/s}, 125\,\mathsf{m/s}, 150\,\mathsf{m/s}, 175\,\mathsf{m/s}, 200\,\mathsf{m/s}\}$; see~\cite{CamachoH2024Investigations} for details of the setup. Each trajectory consists of 700 data points.
The data corresponding to inlet velocities $\{100\,\mathsf{m/s}, 125\,\mathsf{m/s}, 175\,\mathsf{m/s}, 200\,\mathsf{m/s}\}$ form our training data and the data corresponding to inlet velocity $v_{\text{in}}=150\,\mathsf{m/s}$ is our test data. 

\subsubsection{Results}
We compare QMSR to the linear approximations obtained with QDEIM and reconstructions on the quadratic manifold that use all components of data vectors. We show the approximations of the temperature and pressure fields corresponding to time $t = 10^{-3}s$ in Figure~\ref{fig:engine:reconstruction}. The reduced dimension is $\nred = 20$ and the number of sparse sampling points is $\nrows = 40$. The basis matrix $\Vlin$ constructed by the proposed sparse greedy method uses the following left-singular vectors, \begin{align*}
   \Vlin=\left[\leftsing{1}, \leftsing{2}, \leftsing{3}, \leftsing{5}, \leftsing{4}, \leftsing{8}, \leftsing{9}, \leftsing{13}, \leftsing{16}, \leftsing{17}, \leftsing{7},\right. \\\left.\leftsing{20}, \leftsing{11}, \leftsing{124}, \leftsing{125}, \leftsing{99}, \leftsing{29}, \leftsing{175}, \leftsing{151}, \leftsing{69}\right]\,,
   \end{align*}
shares the first few left-singular vectors with the basis selected for the full reconstruction, given by
\begin{align*}
  \Vlin=\left[\leftsing{1},   \leftsing{2},   \leftsing{3},   \leftsing{5},   \leftsing{6},   \leftsing{7},  \leftsing{11},  \leftsing{13},  \leftsing{24},  \leftsing{12},  \leftsing{31},\right.\\\left.  \leftsing{16},  \leftsing{47},   \leftsing{4},  \leftsing{83},  \leftsing{26},  \leftsing{14}, \leftsing{123}, \leftsing{41},  \leftsing{34}\right]\,.
\end{align*}

Comparing the linear approximations given by QDEIM and the quadratic approximations obtained with QMSR, one visually sees that the QMSR approximations are more accurate in the sense that they more accurately capture the shock front compared to the linear approximation. The averaged relative error corresponding to all test data points are plotted in Figure~\ref{fig:engine:rel_errs}. First, notice that QDEIM with $\nrows = \nred$ sparse samples, i.e., without oversampling, leads to unstable approximations; see \cite{PeherstorferDG2020Stability} for more details about instabilities of empirical interpolation. Second, the errors of the approximations obtained with all three methods are significantly higher than in the other examples, which shows that the data are challenging to reduce in this example. The error of the approximations obtained with QMSR from sparse samples shows similar behavior as the errors of the reconstruction that use full data for the encoding, which indicates that even when data are challenging to reduce, QMSR can leverage the higher expressivity of quadratic manifold compared to linear approximations with few sparse samples. A comparison between the sparse linear encoder and the nonlinear encoding is shown in Figure~\ref{fig:engineGN} and is in agreement with the results shown for the previous examples.

\begin{figure}
    \centering
    \resizebox{1.0\textwidth}{!}{\input{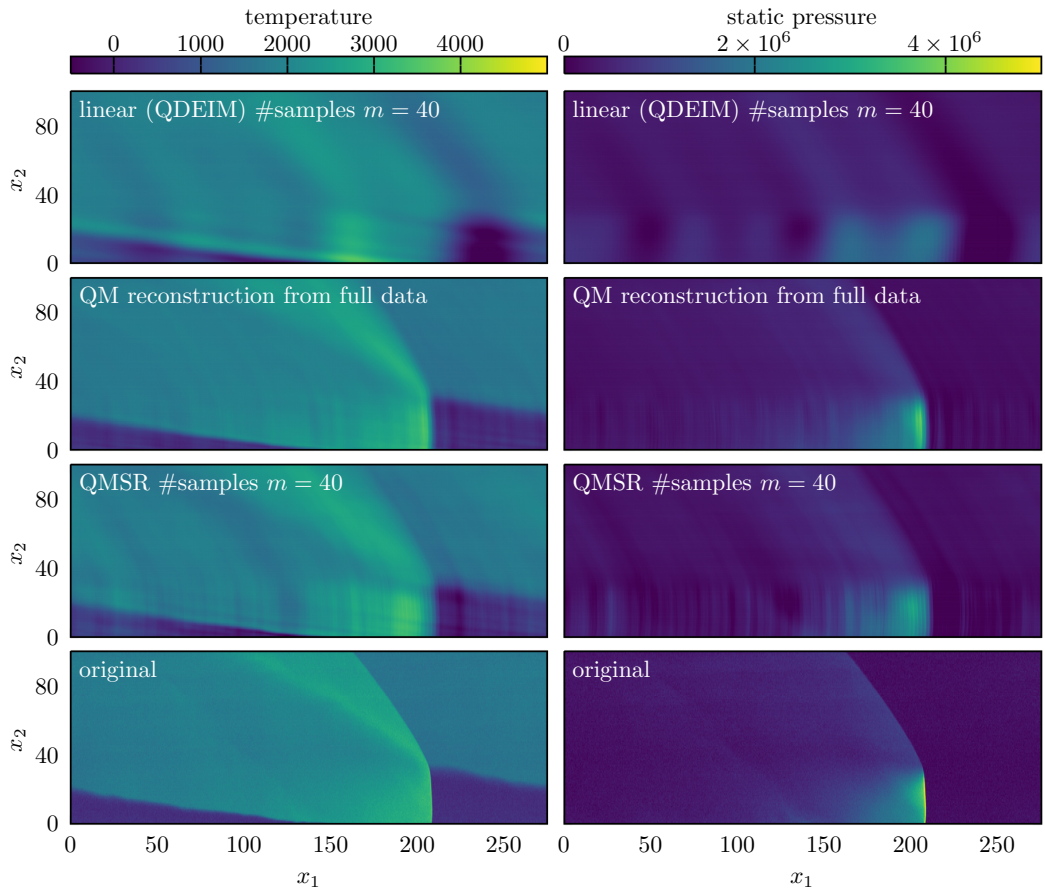}}
    \caption{Rotating detonation engines:
    The QMSR approximations capture more accurately the shock front than the linear approximations obtained with QDEIM. Note that all approximations contain negative temperature values; however, the negative values are more pronounced in the linear approximations.}
    \label{fig:engine:reconstruction}
\end{figure}

\begin{figure}
    \centering
    \resizebox{1.0\textwidth}{!}{\input{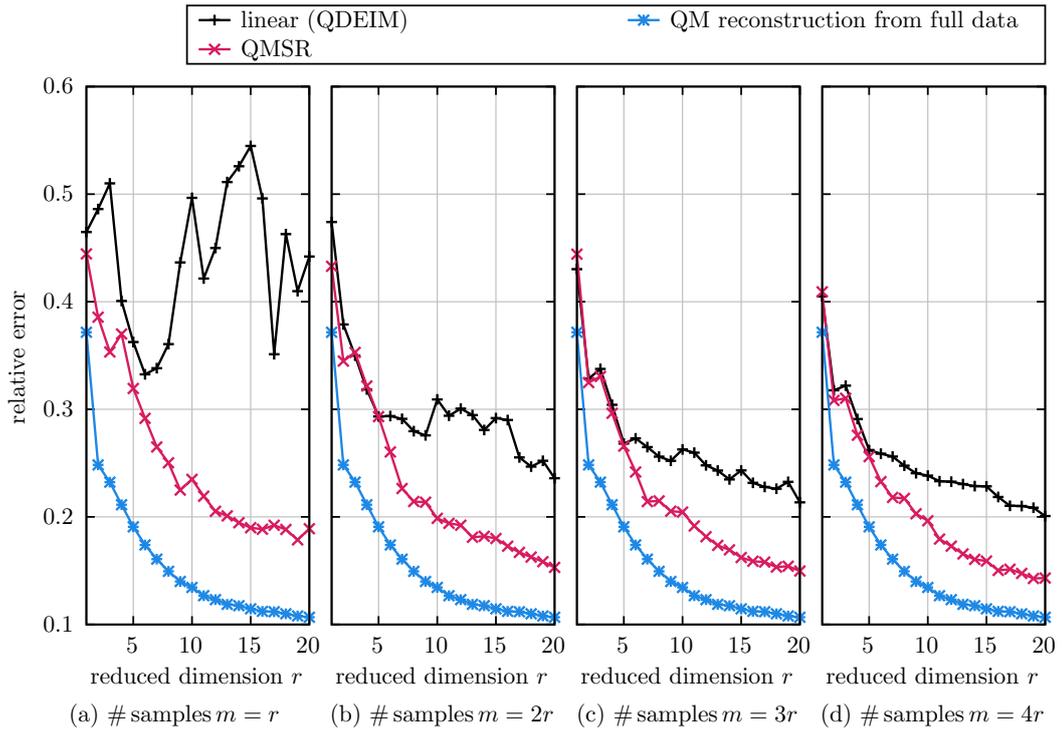}}
    \caption{Rotating detonation engines: The QMSR approximations on quadratic manifolds achieve higher accuracy than linear approximations.}
    \label{fig:engine:rel_errs}
\end{figure}

\begin{figure}
    \centering
    \resizebox{1.0\textwidth}{!}{\input{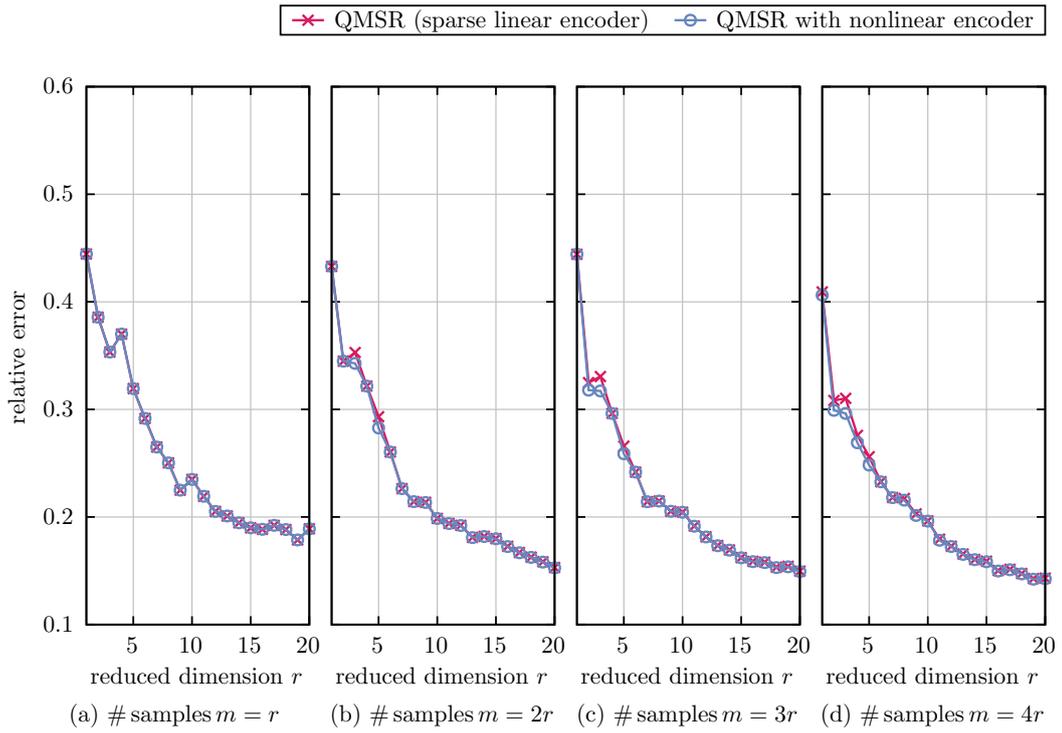}}
    \caption{Rotating detonation engines: The sparse linear encoder of QMSR achieves comparable accuracy as the nonlinear encoding obtained by solving the nonlinear least-squares problem \eqref{eq:ERQM:OptiProb} with the Gauss-Newton method in this example.}
    \label{fig:engineGN}
\end{figure}

\section{Conclusions}\label{sec:Conc}
QMSR combines nonlinear approximations with sparse regression to approximate data vectors from a small subset of their components. Through a judicious choice of the encoder function and training with the proposed sparse greedy method, the QMSR approximations carry over key features from empirical interpolation, such as the ability to exactly recover data that lie on the quadratic manifold from sparse samples. Numerical experiments across a wide range of examples demonstrate that QMSR approximations achieve accuracy comparable to reconstructing data points from all components on quadratic manifolds. Furthermore, QMSR leverages the greater expressivity of quadratic manifolds compared to linear approximation spaces, achieving orders of magnitude higher accuracy than empirical interpolation. 

\section*{Acknowledgments}
We thank Ryan Camacho and Cheng Huang for sharing their dataset from~\cite{CamachoH2024Investigations}.

\bibliographystyle{siamplain}
\bibliography{references,donotchange}

\end{document}